\documentclass[amsart]{amsart}
\usepackage{latexsym}
\usepackage{amssymb}
\usepackage{amsmath}
\usepackage{amsfonts}
\usepackage{amsthm}\theoremstyle{plain}
\newtheorem{theorem}{Theorem}[section]
\newtheorem{corollary}[theorem]{Corollary}
\newtheorem{lemma}[theorem]{Lemma}

\newtheorem{proposition}[theorem]{Proposition}
\theoremstyle{definition}
\newtheorem{definition}[theorem]{Definition}
\theoremstyle{remark}

\renewcommand{\det}{\operatorname{det}}
\newcommand{\perm}{\operatorname{perm}}
\newcommand{\mperm}{\operatorname{mperm}}
\newcommand{\hoperm}{\operatorname{hoperm}}

\newcommand{\rank}{\operatorname{rank}}

\newcommand{\dc}{\operatorname{dc}}

\newcommand{\NN}{\mathbb{N}}

\newcommand{\ee}{\mathbf{e}}
\newcommand{\mm}{\mathbf{m}}
\newcommand{\vv}{\mathbf{v}}
\newcommand{\xx}{\mathbf{x}}

\newcommand{\zerovec}{\mathbf{0}}

\title[Bounds on Determinantal Complexity]{Bounds on Determinantal Complexity of Two Types of Generalized Permanents}
\author[T. Bogart]{Tristram Bogart}
\address{
Tristram Bogart\\
Departamento de Matem\'aticas\\
Universidad de los Andes\\
Carrera $1^{\rm ra}\#18A-12$\\ 
Bogot\'a, Colombia
}
\email{tc.bogart22@uniandes.edu.co}
\author[J.A. Valero]{Juan Andr\'es Valero}
\address{
Juan Andr\'es Valero\\
Departamento de Matem\'aticas\\
Universidad de los Andes\\
Carrera $1^{\rm ra}\#18A-12$\\ 
Bogot\'a, Colombia
}
\email{ja.valeros@uniandes.edu.co}

\thanks{The authors would like to thank Mauricio Velasco for suggesting the idea of multipermanents. The first-named author was supported by internal research grant INV-2020-105-2076 from the Faculty of Sciences of the Universidad de los Andes.}
\date{\today}

\begin{document}

\begin{abstract}
We define two new families of polynomials that generalize permanents and prove upper and lower bounds on their determinantal complexities comparable to the known bounds for permanents. One of these families is obtained by replacing permutations by signed permutations, and the other by replacing permutations by surjective functions with preimages of prescribed sizes.
\end{abstract}

\maketitle{}

\section{Introduction}
The determinantal complexity of a multivariate polynomial $f(x_1, \dots, x_n)$ defined over a field $k$ is the minimum number $m$ for which there exists an $m \times m$ matrix of affine linear functions of $x_1, \dots, x_n$ whose determinant is equal to $f$. A flagship problem in algebraic complexity is to prove or disprove \emph{Valiant's Conjecture}, which states that the determinantal complexity of the \emph{permanent} $\perm_n = \sum_{\sigma \in S_n} \prod_{i=1}^n x_{i,\sigma(i)}$ of an $n \times n$ matrix grows superpolynomially as a function of $n$. The best known upper and lower bounds for this complexity, as discussed below, are respectively exponential and quadratic. 

We will prove similar bounds on two families of generalized permaments. The first family is obtained from permanents by replacing permutations by \emph{signed permutations}: that is, permutations $\tau$ of the set $\{1, \dots, n, -1, \dots, -n\}$ such that $\tau(-i) = -\tau(i)$ for $i=1, \dots, n$. In effect, we replace the symmetric group by the hyperoctahedral group of symmetries of the $n$-cube, and thus these polynomials can be seen as a \emph{type B analogue} \cite{Reiner} of the permanents.

\begin{definition} We define the \emph{hyperoctahedral permanent} $\hoperm_n$ to be the polynomial on $2n^2$ variables $x_{i,j}: 1 \leq i \leq n, 1 \leq \pm j \leq n$ given by 
\[ \hoperm_n = \sum_{\sigma \in S_n} \sum_{(\epsilon_1,\dots,\epsilon_n) \in \{\pm 1\}^n} \prod_{i=1}^n x_{i, \epsilon_i \sigma(i)}\]
\end{definition}

The second family is obtained from permanents by generalizing permutations to functions with preimages of prescribed sizes. 

\begin{definition} Let $\gamma \in \NN$ and $\mm=(m_1,\dots,m_n)$ be a composition of $\gamma$. We define the \emph{multipermament}
  \[ \mperm_\mm = \sum_{\sigma \in \Sigma_\mm} \prod_{i=1}^n x_{i,\sigma(i)} \]
  where $\Sigma_\mm$ is the set of functions $\sigma: [\gamma] \to [n]$ such that $|\sigma^{-1}(i)| = m_i$ for $i=1,\dots,n$.
\end{definition}
We recover the ordinary permanents as $\perm_n = \mperm_{(1,\dots,1)}$. Also, note that reordering the components of $\mm$ simply results in relabeling the variables in the multipermanent. We define multipermanents via compositions rather than partitions in order to facilitate proofs by induction.

It is intuitive from the definitions of hyperoctahedral permanents and multipermanents that these polynomials should be at least as hard to calculate as permanents. We could make this intution precise by using the notion of \emph{VNP-completeness}, introduced by Valiant \cite{valiant1}. Informally, (sequences of) permanents, hyperoctahedral permanents, or multipermanents are all in the class VNP because any given coefficient is easy to determine. Valiant showed that the family of permanents is in fact VNP-complete; that is, if there is a family of polynomial-size algebraic circuits to compute permanents, then VP=VNP, and in fact P=NP \cite{valiant2}. Now $\perm_n$ can be efficiently reduced either to $\hoperm_n$ or to $\mperm_\mm$ where $\mm$ is any composition into $n$ nonzero parts, so families of hyperoctahedral permanents or of multipermanents with $n$ increasing are also VNP-complete. A recent article by Ikenmeyer and Landsberg \cite{IL} lays out precise relations between VNP-completeness, determinantal complexity, and various other measures of algebraic complexity. 

Meaningful upper and lower bounds on determinantal complexity are typically difficult to obtain. The best known upper bound on $\dc(\perm_n)$ for $n \geq 3$, obtained by Grenet \cite{Grenet}, is $2^{n} - 1$. Our new polynomials are suggested by the observation that Grenet's method of proof works for any \emph{poset polynomial} (as defined in the next section). For permanents, the upper bound is obtained from a natural labeling of the cover relations of the Boolean lattice. By using similar labelings of the face lattice of the $n$-cube, and of the lattice of multisets contained in a given multiset (equivalently, of monomials dividing a given monomial), we obtain the following upper bounds. 
\begin{theorem} \label{thm:upperbounds} $ $
  \begin{enumerate} 
  \item For every $n$, we have $\dc(\hoperm_n) \leq 3^n$.
  \item For every composition $\mm=(m_1,\dots,m_n)$, we have \\
    $\dc(\mperm_\mm) \leq \left( \prod_{i=1}^n (m_i+1) \right) - 1$. 
  \end{enumerate}
\end{theorem}

For $n \geq 5$, the best known lower bound on $\dc(\perm_n)$ in characteristic zero, obtained by Mignon and Ressayre \cite{MR}, is $\frac{n^2}{2}$. Their proof involves bounding the rank of the Hessian matrix at a given zero of the permanent. Cai, Chen, and Li \cite{CCL} later adapted this technique to give a quadratic lower bound for $\dc(\perm_n)$ in characteristic $p > 2$. The technique can also be applied to other polynomials. For example,  Chen, Kayal, and Wigderson suggest as an exercise \cite[Exercise 13.6]{CKW} to show in this way that the determinantal complexity of the symmetric polynomial $\sum_{1 \leq i \neq j \leq n}x_ix_j$ is $\Omega(n)$. In a more involved application, Chillara and Mukhopadhyay \cite{CM} showed that the determinantal complexity of the \emph{iterated matrix multiplication polynomial}
\[IMN_{n,d}(X)=\sum\limits_{i_1,\ldots,i_{d-1}\in [n]} x_{1,i_1}^{(1)}x_{i_1,i_2}^{(2)}\cdots x_{i_{d-2},i_{d-1}}^{(d-1)}x_{i_{d-1},1}^{(d)}\]
(that is, the top left entry of the product of $d$ generic $n \times n$ matrices) is at least $d(n-1)$. 

Unlike the Grenet argument for upper bounds, each application of the Mignon-Ressayre method to a new family of polynomials seems to require significant work: first in identifying an appropriate zero of the polynomial (not all zeros yield a useful bound), then in analyzing the Hessian at this zero. In this way, we prove the following results. 

\begin{theorem} \label{thm:lowerbounds} Let $n \geq 3$ and $k$ be a field of characteristic zero.
  \begin{enumerate}
  \item For all $n$ we have $dc(\hoperm_n)\geq n^2/2.$
  \item For all $\mm \in \NN^n$ we have $\dc(\mperm_\mm) \geq (m_1 + \dots + m_n)^2 / 2$. 
    \end{enumerate}
\end{theorem}

We note that these upper and lower bounds are exponentially far apart, just as in the case of the permanent. It would be interesting to identify new families of polynomials based on combinatorially natural posets for which the approach described here not only succeeds in proving both lower and upper bounds, but for which the upper and lower bounds are significantly closer together.   

The rest of the paper is organized as follows. In Section 2, we explain Grenet's method as applied to general labeled graded posets and apply it to prove Theorem \ref{thm:upperbounds}.
In Section 3, we explain the Mignon-Ressayre method and apply it to prove Theorem \ref{thm:lowerbounds}.

\section{Proofs of upper bounds and VNP-completeness}
Let $P$ be a graded poset of rank $d$ with $p$ elements, including a unique maximum $\hat{1}$ and a unique minimum $\hat{0}$. We say that a polynomial $f = f(x_1, \dots, x_n)$ is \emph{supported on $P$} if there exists a labelling of the edges $e$ of the Hasse diagram of $P$ with linear forms $L_e$ in the variables $x_1, \dots, x_n$ such that
\[ f = \sum_{C} \left( \prod_{e \in C} L_e \right) \]
where $C$ ranges over all of the saturated chains $\hat{0} < i_1 < i_2 < \dots < i_{d-1} < \hat{1}$ in $P$. 

\begin{proposition} \label{prop:poset_bound} If $f$ is supported on a graded poset $P$ with $p$ elements and a unique minimum $\hat{0}$ and maximum $\hat{1}$, then $f$ has a determinantal representation of size $p-1$.
\end{proposition}

  \begin{proof}
    The construction of polynomials supported on a poset directly generalizes Grenet's representation of the permanent polynomial \cite{Grenet} and the proof will do so as well. Recall that a \emph{cycle cover} in a directed graph $G$ is a union of vertex-disjoint cycles that passes through every vertex of $G$, and that the permanent of the directed adjacency matrix equals the number of cycle covers of $G$. If we replace the 1's in the adjacency matrix by arbitrary weights, then the permanent computes the number of weighted cycle covers.

    To obtain our graph $G$, we begin with the Hasse diagram of $P$. That is, each cover relation $x \lessdot y$ (i.e. edge of the Hasse diagram) defines an edge $e = x \to y$ in $G$ which we label by $L_e$. Now identify the maximum of $P$ with the minimum as a single vertex $v_0$. Finally, we add a loop at every vertex except $v_0$ and label these loops by 1.

Now let $A$ be the weighted adjacency matrix of $G$ with weights given by the labeling $L_e$. Every cycle cover passes through $v_0$ exactly once, so it includes one cycle of the form
    $(v_0, v_{i_1}, v_{i_2}, \dots, v_{i_{d-1}}, v_0)$ where $C = \hat{0} \lessdot i_1 \lessdot i_2 \lessdot \dots \lessdot i_{d-1} \lessdot \hat{1}$ is a saturated chain in $P$. To avoid passing through $v_0$ again, we must complete the cycle cover by taking the loop at each vertex outside $C$. Thus
\[ \perm(A) = \sum_{C} \left( \prod_{e \in C} L_e \right) 1^{p-1-d} = f . \]

    But since $P$ is graded, all of the cycles have the same length $d$, so all of the cycle covers have the same sign (positive if $d$ is odd, or negative if $d$ is even.) Thus   
    \[ \det(A) = \pm \perm(A) = \pm f. \]

    Finally, if this process yields $-f$ instead of $f$, we obtain a determinantal representation of $f$ by multiplying any single row of $A$ by -1.
  \end{proof}

  Grenet's determinantal representation of $perm_n$  of size $2^n -1$ is obtained by applying this process to the Boolean lattice on $n$ elements. If $T$ is a set of size $i$ and $e$ is the edge from the set $T$ to the set $T \cup \{j\}$, then we take $L_e = x_{ij}$.  

  \begin{proof}[Proof of Theorem \ref{thm:upperbounds}] For the multipermanent $\mperm_{\mm}$, let $P_\mm$ be the poset of multisets contained in the multiset $S = \{1^{m_1},\dots,n^{m_n}\}$. Then $P_\mm$ is a graded poset of rank $\gamma = m_1 + \dots + m_n$ with $\prod_{i=1}^n (m_i+1)$ elements, including a unique maximum $\emptyset$ and a unique maximum $S$. Just as in the case of permanents, if $T$ is a multiset of size $i$ and if there is an edge from $T$ to $T \cup \{j\}$, then we label this edge by $x_{ij}$. Each term of the multipermanent indexes a unique saturated chain for $\emptyset$ to $S$, so the result follows from Proposition \ref{prop:poset_bound}.  

    For the hyperoctahedral permanent $\hoperm_n$, let $Q_n$ be the face lattice of the $n$-cube. A face $F$ of the cube determines a unique vector $u=(u_1,\dots,u_n) \in \{0,1,-1\}^n$ where $u_i = 1$ if $F$ is contained in the hyperplane $y_i=1$, $u_i = -1$ if $F$ is contained in the hyperplane $y_1 = -1$, and otherwise $u_i=0$. Conversely, every $u \in \{0,1,-1\}^n$ determines a unique face of the cube, so $\tilde{Q}$ is a graded poset of rank $n$ with $3^n$ elements including a unique minimum $\emptyset$ and $2^n$ maxima indexed by vectors $u \in \{1,-1\}^n$. The covering relation is given by $u \lessdot v$ when:
     \begin{enumerate}
     \item there is a unique index $j$ such that $u_j = 0$ and $v_j \neq 0$, and
     \item $u_k = v_k$ for all $k \neq j$.
     \end{enumerate}  
     We label such a relation (i.e. edge in the Hasse diagram) by $x_{i, \pm j}$ where $i$ is the rank (= number of nonzero coordinates) of $u$ and the sign depends on $v_j$. The result then follows from the definition of hyperoctahedral permanents and Proposition \ref{prop:poset_bound}. 
  \end{proof}

\section{Proofs of lower bounds}
Let $k$ be a field of characteristic zero. For a polynomial $f \in k[x_1, \dots, x_n]$, we consider the gradient function 
\begin{align*}
    Tf: k^N & \longrightarrow k^N \\
    x& \longmapsto T_xf=\left(\frac{\partial f}{\partial x_1}(x),\ldots,\frac{\partial f}{\partial x_N}(x)\right)
\end{align*}
and the Hessian 
\begin{align*}
    T^2f:k^N & \longrightarrow k^{N \times N}  \\
    x& \longmapsto T^2_xf=\left(\frac{\partial^2 f}{\partial x_i\partial x_j}(x)\right)_{1\leq i,j \leq N}.
\end{align*}

\begin{lemma} \cite[Proposition 3.8]{MR} \label{lem:Hessian} 
  For any singular $n \times n$ matrix $A$, we have \\
  $\rank (T^2_A  {\det}_{n})\leq 2n.$
\end{lemma}

The following result is also based on \cite{MR}, but appears there only for permanents. 

\begin{proposition} \label{prop:Hessian}
Given $f\in k[x_1,\ldots,x_m]$, let $F:k^m\rightarrow M_n(k)$ be a determinantal repesentation of $f$; that is, $f$ is an affine linear function such that $f=\det_n\circ F$. If $f(y)=0$ then
$\rank (T^2_y f) \leq 2n.$ 
\end{proposition}
\begin{proof}
Let $F(x)=(g_{ij}(x))_{i,j\in [n]}$. Since $F$ is affine linear, all of its partial derivatives are constant so write $c_{ij,h} = \frac{\partial g_{ij}}{\partial x_h}$. By two applications of the chain rule, we obtain
\[ \frac{\partial f}{\partial x_h}(y) = \sum\limits_{i,j\in [n]}c_{ij,h}\frac{\partial {\det}_{n}}{\partial z_{i,j}}(F(y)), \]
\[ \frac{\partial^2 f}{\partial x_h\partial x_e}(y) = \sum\limits_{i,j\in [n]}\sum\limits_{k,l\in [n]}c_{ij,h}\frac{\partial^2 {\det}_{n}}{\partial z_{k,l}\partial z_{i,j}}(F(y)) c_{kl,e}.\]
Thus $T^2_y f=L T^2_{F(y)}{\det}_{n} L^t$ where 
\[L=\left(\begin{array}{cccc}
    c_{11,1} & c_{12,1} & \cdots & c_{nn,1} \\
    \vdots & \vdots & \ddots & \vdots \\
    c_{11,m} & c_{12,m} & \cdots & c_{nn,m}
\end{array}\right)_{m\times n^2}.\]
Thus $\rank (T^2_y f)= \rank \left( L T^2_{F(y)}{\det}_{n} L^t\right) \leq \rank (T^2_{F(y)}{\det}_{n})\leq 2n$, where the last inequality follows from Lemma \ref{lem:Hessian} and the fact that ${\det}_{n}(F(y))=f(y)=0.$
\end{proof}

The following general bound on determinantal complexity is an immediate corollary of Proposition \ref{prop:Hessian}.
\begin{corollary} \label{cor:Hessian_dc_bound}
Let $f\in k[x_1,\ldots,x_m]$. For each $y\in k^m$ such that $f(y)=0$, 
\[\rank(T^2_y f)/2\leq \dc(f).\]
\end{corollary}

For our lower bounds, we will need the following fact from linear algebra.
\begin{lemma} \cite[Lemma 3.7]{MR} \label{lem:off_diagonal}
  Let $a, b \in \NN$ and let $Q,R$ be $a \times a$ invertible matrices. Then the $(ab)\times(ab)$ matrix
  $M=\left(\begin{array}{ccccc}
     0 & Q & Q & \cdots & Q \\
     Q & 0 & R & \cdots & R \\
     Q & R & 0 & \cdots & R \\
     \vdots & \vdots & \vdots & \ddots & \vdots \\
     Q & R & R & \cdots & 0
\end{array} \right)$
  is invertible.
\end{lemma}

\subsection{Proof of the lower bound for hyperoctahedral permanents}
 Recall that $\hoperm_n$ is a polynomial in $2n^2$ variables $\{ x_{i,j}: 1 \leq i \leq n, 1 \leq \pm j \leq n\}$. We take these variables to form an $(n \times 2n)$ generic matrix
\[ X = \begin{bmatrix}
    x_{1,1} & \cdots & x_{1,n} & x_{1,-1} & \cdots & x_{1,-n} \\
    \vdots &  & \ddots        &          &       & \vdots \\
    x_{n,1} & \cdots & x_{n,n} & x_{n,-1} & \cdots & n_{1,-n}
\end{bmatrix}.\]

We begin by proving a recursive formula for the hyperoctahedral permanent using row expansion.
\begin{lemma}\label{lem:hoperm_recurrence}
  Let $n \geq 2$ and $X = (x_{i,j})$ be as above. Then
  \[ \hoperm_n(X) = \sum_{j=1}^n \left(x_{i,j} + x_{i,-j} \right) \hoperm_{n-1}(X_{i, \pm j}) \]
  where $X_{i, \pm j}$ is obtained from $X$ by removing row $i$ and the columns indexed by $j$ and $-j$. 
\end{lemma}
\begin{proof}
  \begin{align*}  \hoperm_n &= \sum_{\sigma \in S_n} \sum_{(\epsilon_1,\dots,\epsilon_n) \in \{\pm 1\}^n} \prod_{k=1}^n x_{k, \epsilon_k \sigma(k)} \\
    &= \sum_{j=1}^n \sum_{\substack{\sigma \in S_n \\ \sigma(i)=j}} (x_{i,j} + x_{i,-j}) \sum_{\epsilon_1,\dots,\hat{\epsilon_i},\dots,\epsilon_n \in \{\pm 1\}^{n-1}}  \prod_{k \neq i}x_{k,\epsilon_k \sigma(k)} \\
    &= \sum_{j=1}^n \left(x_{i,j} + x_{i,-j} \right) \hoperm_{n-1}(X_{i, \pm j}).
    \qedhere
    \end{align*}
\end{proof}

In order to identify an appropriate zero, we need the following observation. Let $U_{n,m}$ be the $n \times m$ matrix of all ones.

\begin{lemma}\label{lem:hoperm_ones}
For all $n$, we have $\hoperm_n(U_{n,2n}) = 2^n n!$. 
\end{lemma}
\begin{proof}
For $n=1$ we have $U_{1,2} = (1, 1)$ so the hyperoctedral permanent is $1+1 = 2^1 1!$. For $n \geq 2$, we use Lemma \ref{lem:hoperm_recurrence} and induction to obtain
\[\hoperm_{n+1}(U_{n+1,2n+2})=\sum\limits_{j=1}^{n+1} 2 \cdot \hoperm_{n}(U_{n,2n}) =\sum\limits_{j=1}^{n+1} 2 \cdot 2^{n}n! =2^{n+1}(n+1)!. \qedhere\]
\end{proof}

\begin{proposition}\label{prop:hoperm_zero1}
Consider the $n \times 2n$ matrix $B=(b_{i,j})$ where
\[b_{i,j} = \begin{cases}  -2n+1 & \mbox{if } i=n \ \mbox{and} \ j=n \\ 1 & \mbox{otherwise. }  \end{cases}\]
Then $\hoperm_{n}(B)=0$.
\end{proposition}
\begin{proof}
The proof is again by induction on $n$. If $n = 1$, then $B = (-1, 1)$ and so $\hoperm(B) = -1+1=0$. If $n \geq 2$, then from Lemma \ref{lem:hoperm_recurrence} and Lemma \ref{lem:hoperm_ones} we obtain
\begin{align*}
\hoperm_{n}(B)&= \sum_{j=1}^{n-1} 2 \hoperm_{n-1}(U_{n-1,2n-2})+(-2n+1+1)\hoperm_{n-1}(U_{n-1,2n-2}) \\
&=2(n-1)2^{n-1}(n-1)!+(-2n+2)2^{n-1}(n-1)! \\
&=0.
\qedhere
\end{align*}
\end{proof}
We now analyze the Hessian matrix of $\hoperm_n$ at the zero $B$. Let $W_n=U_{(n,n)}-I_n$.

\begin{proposition}\label{prop:hessianhoperm}
  Order the variables of $\hoperm$ as follows:
  \[a_{1,1},\ldots,a_{1,n},\ldots,a_{n,1},\ldots,a_{n,n},a_{1,-1},\ldots,a_{1,-n},\ldots,a_{n,-1},\ldots,a_{n,-n}.\]
  Then the Hessian matrix of $\hoperm$ evaluated at $B$ is 
  \[ H(B)=\left(\begin{array}{cc}
    C & C \\
    C & C
\end{array}\right)\]
where $C$ is the $n^2\times n^2$ matrix given in $n \times n$ blocks by
\[C=2^{n-2}(n-3)!\left(\begin{array}{ccccc}
     0 & A  & \cdots & A & (n-2)W\\
     A & 0  & \cdots & A & (n-2)W\\
     \vdots & \vdots & \ddots & \vdots & \vdots \\
     A & A &  \cdots & 0 & (n-2)W\\
     (n-2)W & (n-2)W & \cdots & (n-2)W & 0
\end{array}\right),\]
\[\textup{where } A=\left(\begin{array}{ccccc}
     0 & -2  & \cdots & -2 & n-2\\
     -2 & 0  & \cdots & -2 & n-2\\
     \vdots & \vdots & \ddots & \vdots & \vdots \\
     -2 & -2 &  \cdots & 0 & n-2\\
     n-2 & n-2 & \cdots & n-2 & 0
\end{array}\right).\]
\end{proposition}

\begin{proof}
By the symmetry between positive and negative indices in $\hoperm_n$, we have
  $H_{(i,\pm j),(k, \pm l)} = H_{(i,j),(k, l)}$. Thus $H$ can be partitioned into four identical square blocks, each of which we label $C$. From now on we thus assume that $j$ and $l$ are positive.
  
  We now consider the distinct cases for the indices $i$ and $k$.
  \begin{enumerate}
  \item If $i=k$, then $H_{(i,j),(k,l)}=0$, regardless of $j$ and $l$. This justifies the diagonal blocks of zeros in $C$. 
  \item Suppose $i\neq k$ and $i,k<n$. We then consider different cases for $j$ and $l$. 
    \begin{enumerate}
        \item If $j=l$, then again $H_{(i,j),(k,l)}=0$.
        \item If $j\neq l$ and neither is equal to $n$, then 
        \begin{align*}
            H_{(i,j),(k,l)}&=(-2n+2)\hoperm_{n-3}(U_{n-3,2n-6})+\sum\limits_{t\neq j,l,n} 2\hoperm_{n-3}(U_{n-3,2n-6}) \\
            &=(-2n+2)2^{n-3}(n-3)!+2(n-3)\cdot2^{n-3}(n-3)! \\
            &=-2^{n-1}(n-3)!.
        \end{align*}
        \item If $j\neq l$ and $j$ or $l$ equals $n$, then 
\[ H_{(i,j),(k,l)}=\sum\limits_{t\neq j,l} 2\hoperm_{n-3}(U_{n-3,2n-6}) =2(n-2)2^{n-3}(n-3)! =2^{n-2}(n-2)!.\]
    \end{enumerate}
    This justifies the blocks labelled $A$. 
  \item Finally, suppose $i\neq k$ and either $i$ or $k$ is equal to $n$. If $j = l$ then we still have $H_{(i,j),(k,l)}=0$. Otherwise, we have \[ H_{(i,j),(k,l)}= \hoperm_{n-2}(U_{n-2,2n-4})=2^{n-2}(n-2)!.\] This justifies the blocks labelled $W$. \qedhere
\end{enumerate}
\end{proof}

\begin{proof}[Proof of Theorem \ref{thm:lowerbounds}(2)]
By Lemma \ref{lem:off_diagonal} we have that the matrices $A$ and $(n-2)W$ are nonsingular. Again using Lemma \ref{lem:off_diagonal}, it now follows that the $n \times n$ matrix $C$ is also nonsingular, so the rank of $H$ is $n$. The theorem now follows from Corollary \ref{cor:Hessian_dc_bound} and Proposition \ref{prop:hessianhoperm}.
\end{proof}

\subsection{Proof of the lower bound for multipermanents}
Let $\gamma \in \NN$ and $\mm = (m_1,\dots,m_n)$ be a composition of $\gamma$. As $\mperm_\mm$ is a polynomial in the variables $(x_{ij})$ for $1 \leq i \leq \gamma, 1 \leq j \leq n$, we consider its input to be a $\gamma \times n$ matrix.

Again, we begin with a recurrence for multipermanents and their values at matrices of all ones.

\begin{lemma} \label{lem:mperm_recurrence}
  Let $\mm \in \NN^n$. Then \[\mperm_\mm = \sum_{j=1}^n x_{\gamma, j} \mperm_{\mm - \ee_j}.\]
  where $\ee_j$ is the $j$th standard basis vector and we take $\mperm_{\mm'}$ to be the zero polynomial if $\mm'_j$ is negative for some $j$. 
 \end{lemma}
\begin{proof}
By definition, $\mperm_{\mm - \ee_j}$ is the number of chains in $P_\mm$ from $\emptyset$ to the multiset $\{1^{m_1}, \dots, i^{m_{i}-1}, \dots, n^{m_n}\}$. The total number of chains from $\emptyset$ to $S$ is obtained by summing over the possible element $j$ to be added at the last step.  
\end{proof}

\begin{lemma} \label{lem:mperm_ones} 
For all $\gamma$ and $n$, we have $\mperm_\mm(U) = \binom{\gamma}{m_1, m_2, \dots, m_n} = \frac{\gamma!}{m_1! m_2! \dots m_n!}$. 
\end{lemma}
\begin{proof}
Note that $\mperm_\mm(U)$ is simply the number of terms of $\mperm_\mm$, which is the number of saturated chains in the poset $P_\mm$. Such a chain is given by a sequence $a_1, \dots, a_\gamma$ in which the number $j$ appears exactly $m_j$ times, and $\binom{\gamma}{m_1, m_2, \dots, m_n}$ is by definition the number of such sequences.     
\end{proof}

Again, we begin by identifying an appropriate zero of $\mperm_\mm$. There are two cases for the zero we will choose, depending on whether or not $m_1, \dots, m_n$ are all equal. 

\begin{proposition} \label{prop:mperm_zero1} Let $\mm \in \NN^n$.
  \begin{enumerate}
  \item If $m_1, \dots, m_n$ are not all equal, then let $k$ be such that $m_k = \max\{m_1, \dots, m_n\}$, $C = \{j: m_j = m_k \}$, and $c=|C|$. Consider the $\gamma \times n$ matrix $B^{(1)} = (b_{i,j})$ where 
 \[b_{i,j} = \begin{cases} 1-\frac{\gamma}{c m_k} & \mbox{ if }i=\gamma \mbox{ and }  j \in C \\
    1 & \mbox{ otherwise }. \end{cases}\]
\item If $m_1 = \cdots = m_n$, then consider the $\gamma \times n$ matrix $B^{(2)} = (b_{i,j})$ where 
  \[b_{i,j} = \begin{cases} 1-n & \mbox{ if }i=1 \mbox{ and }  j =1 \\
    1 & \mbox{ otherwise }. \end{cases}\]
  \end{enumerate}
  In each case, the chosen matrix $B^{(\ell)}$ is a zero of $\mperm_\mm$.
\end{proposition}

\begin{proof}
 \begin{enumerate}
  \item  Using Lemmas \ref{lem:mperm_ones} and \ref{lem:mperm_recurrence}, we obtain that $\mperm_\mm\left(B^{(1)}\right)$ equals
    \begin{align*}
      &  \sum_{j\notin C}  1 \cdot \mperm_{\mm - \ee_j}(U_{\gamma-1,n-1})  + \sum_{j\in C}  \left(1-\frac{\gamma}{c \cdot m_k}\right) \cdot \mperm_{\mm - \ee_j}(U_{\gamma-1,n-1})   \\
   &= \sum_{j\notin C}  \binom{\gamma-1}{m_1,\dots,m_j-1, \dots, m_n}  + \sum_{j\in C}  \left(1-\frac{\gamma}{c \cdot m_k}\right) \cdot  \binom{\gamma-1}{m_1,\dots,m_j-1, \dots, m_n}  \\ 
   &=  \sum_{j \notin C} \frac{m_j}{m_k}\binom{\gamma-1}{m_1,\dots,m_k-1, \dots, m_n}  + \left(1-\frac{\gamma}{c \cdot m_k}\right) c\binom{\gamma-1}{m_1,\dots,m_k-1, \dots, m_n} \\
   &=  \binom{\gamma-1}{m_1,\dots,m_k-1, \dots, m_n}\left(\frac{1}{m_k}\left(\gamma-\sum_{j \in C} m_j\right)+ \left(1-\frac{\gamma}{c \cdot m_k}\right) c\right) \\
   &=  \binom{\gamma-1}{m_1,\dots,m_k-1, \dots, m_n}\left(\frac{1}{m_k}\left(\gamma-c \cdot m_k\right)+ \left(1-\frac{\gamma}{c \cdot m_k}\right) c\right) \\ 
   &= 0, \end{align*}
 where the third equality follows because
  \[ \frac{m_j}{m_i} \binom{\gamma-1}{m_1,\dots,m_i-1, \dots, m_n}= \binom{\gamma-1}{m_1,\dots,m_j-1, \dots, m_n}\]
  whenever $m_i$ and $m_j$ are both positive, and the remaining terms in the sum are all zero.   
\item By Lemmas \ref{lem:mperm_ones} and \ref{lem:mperm_recurrence}, we obtain
    \begin{align*}
    \mperm_\mm\left(B^{(2)}\right) &=(1-n)\mperm_{\mm - \ee_1}(U_{mn-1,n-1})+ \sum_{j=2}^n 1 \cdot \mperm_{\mm - \ee_j}(U_{mn-1,n-1}) \\
    &=(1-n)\binom{mn-1}{m-1, \dots, m}+ \sum_{j=2}^n \binom{mn-1}{m,\dots,m-1, \dots, m} \\
    &=(1-n)\frac{(mn-1)!}{(m!)^{n-1}
    (m-1)!}+\sum\limits_{j=2}^n \frac{(mn-1)!}{(m!)^{n-1}
    (m-1)!} \\
    &=0. \qedhere
\end{align*}
\end{enumerate}
\end{proof}
Now we need to calculate the Hessian matrix of the multipermanent. Given a matrix $X$, let $X_i$ be the matrix obtained from $X$ by removing row $i$ and let $X_{ii'}$ be the matrix obtained from $X$ by removing rows $i$ and $i'$.

\begin{lemma} \label{lem:mperm_derivates}
Let $i,i' \in [\gamma]$ with $i\neq i'$, and $j,j' \in [n]$. Then
\begin{align*}
\frac{\partial \mperm_{\mm}}{\partial x_{i,j}}(X) &=\mperm_{\mm-\ee_j}(X_i), \\
\frac{\partial^2 \mperm_{\mm}}{\partial x_{i',j'} \partial x_{i,j}}(X) &=\mperm_{\mm-\ee_j-\ee_{j'}}(X_{ii'}).
\end{align*}
\end{lemma}

\begin{proof}
Using Lemma \ref{lem:mperm_recurrence}, we have
 \begin{align*}
\mperm_\mm &= \sum_{k=1}^n x_{i, k} \mperm_{\mm - \ee_k}.
\end{align*}
Differentiating both sides with respect to $x_{ij}$, we obtain 
\[ \frac{\partial \mperm_\mm}{\partial x_{ij}}(X) = \frac{\partial}{\partial x_{ij}}\sum_{k=1}^n x_{i, k} \mperm_{\mm - \ee_k}(X_i) = \mperm_{\mm-\ee_j}(X_i).\]
Again, using Lemma \ref{lem:mperm_recurrence} and then differentiating with respect to $x_{i'j'}$, we obtain
\[\mperm_{\mm-\ee_j}(X_i)=\sum_{k=1}^n x_{i', k} \mperm_{\mm - \ee_j-\ee_k}(X_{ii'}),\]
and so
\[\frac{\partial^2 \mperm_\mm}{\partial x_{i'j'}\partial x_{ij}}(X) = \frac{\partial}{\partial x_{i'j'}}\mperm_{\mm-\ee_j}(X_i) = \mperm_{\mm - \ee_j-\ee_{j'}}(X_{ii'}). \qedhere \]
\end{proof}
We note that if $i=i'$, then $\frac{\partial^2 \mperm_\mm}{\partial x_{i'j'}\partial x_{ij}}(X)=0$.

In order to simplify our analysis, we will from now on assume that $\mm$ is a partition. In the case that not all components are equal, we thus have that $m_1 = \cdots = m_c > m_{c+1 }\geq \cdots \geq m_n$. Let $d=\frac{\gamma}{c \max_i m_i}$. 

\begin{proposition}\label{prop:hessianperm1} Let $\mm$ be a partition and set $\ell = 2$ if all components of $\mm$ are equal or 1 otherwise. The Hessian matrix $H$ of $\mperm_\mm$ evaluated at the zero $B^{(\ell)}$ is of the block diagonal form
 \[ H\left(B^{(\ell)}\right)=\left(\begin{array}{ccccc}
     0 & R^{(\ell)}  & \cdots & R^{(\ell)} & Q^{(\ell)}\\
     R^{(\ell)} & 0  & \cdots & R^{(\ell)} & Q^{(\ell)}\\
     \vdots & \vdots & \ddots & \vdots & \vdots \\
     R^{(\ell)} & R^{(\ell)} &  \cdots & 0 & Q^{(\ell)}\\
     Q^{(\ell)} & Q^{(\ell)} & \cdots & Q^{(\ell)} & 0
 \end{array}\right)_{\gamma n \times \gamma n} ,\]
 where 
  \[Q^{(1)}=k_1 \left(\begin{array}{cccc}
      m_1(m_1-1) & m_1m_2 & \cdots & m_1m_n \\
      m_1m_2 & m_2(m_2-1) & \cdots & m_2m_n \\
      \vdots & \vdots & \ddots & \vdots \\
      m_1m_n & m_2m_n & \cdots & m_n(m_n-1) 
  \end{array}\right)_{n \times n},\]\[ R^{(1)}=k_1\left(\begin{smallmatrix}
2m_1(m_1-1)(d-1)  & \cdots & 2m_1m_{c}(d-1) & m_1m_{c+1}(d-2) & \cdots & m_1m_n(d-2) \\
\vdots  & \ddots & \vdots & \vdots  & \ddots & \vdots \\
2m_cm_1(d-1)  & \cdots & 2m_c(m_c-1)(d-1) & m_cm_{c+1}(d-2)  & \cdots & m_cm_n(d-2) \\
m_{c+1}m_1(d-2)  & \cdots & m_{c+1}m_c(d-2) & -2m_{c+1}(m_{c+1}-1)  & \cdots & -2m_{c+1}m_{n} \\
\vdots & \ddots & \vdots & \vdots & \ddots & \vdots \\
m_nm_1(d-2) & \cdots & m_nm_c(d-2) & -2m_nm_{c+1}  & \cdots & -2m_{n}(m_n-1) 
\end{smallmatrix}\right)_{n\times n},\]
\[ Q^{(2)}=k_2(mn-2) \left(\begin{array}{cccc}
      m(m-1) & m^2 & \cdots & m^2 \\
      m^2 & m(m-1) & \cdots & m^2 \\
      \vdots & \vdots & \ddots & \vdots \\
      m^2 & m^2 & \cdots & m(m-1) 
\end{array}\right)_{n\times n},\]
\[ R^{(2)}=k_2m\left(\begin{array}{ccccc}
     2(m-1)(n-1) & m(n-2) & m(n-2) & \cdots & m(n-2) \\
     m(n-2) & -2(m-1) & -2m & \cdots & -2m \\
     m(n-2) & -2m & -2(m-1) & \cdots & -2m \\
     \vdots & \vdots & \vdots & \ddots & \vdots \\
     m(n-2) & -2m & -2m & \cdots & -2(m-1)
  \end{array}\right)_{n\times n},\]

with $k_1=\tfrac{(\gamma-2)!}{m_1!\cdots m_n!}$ and $k_2=\frac{(mn-3)!}{m!^n}$.
\end{proposition}

\begin{proof}
  The zero blocks occur because $\frac{\partial^2}{\partial x_{i,j}\partial x_{i,j'}} \mperm_\mm = 0$ for any $i$, $j$ and $j'$, as noted above. So for the remainder of the proof we assume in all cases that $i \neq i'$. 

  Consider the case that not all components of $\mm$ are equal. If $i = \gamma$ or $i' = \gamma$, then by the definition of $B^{(1)}$ and Lemma \ref{lem:mperm_derivates} we have
  \begin{align*}
            H_{(i,i'),(j,j')}\left(B^{(1)}\right)&=\mperm_{\mm-\ee_j-\ee_{j'}}(B^{(1)}_{ii'}) \\
            &=\mperm_{\mm-\ee_j-\ee_{j'}}(U_{(\gamma-2,n)}) \\
            &=k_1\left\{ \begin{array}{lcc}
             m_j m_{j'} &  if  & j\neq j' \\
             m_j(m_j-1) &  if & j= j'
             \end{array}\right.         
\end{align*}
  This justifies the matrix $Q^{(1)}$. On the other hand, if $i, i' < \gamma$ then
  \begin{align*}
            H_{(i,i'),(j,j')}\left(B^{(1)}\right)
            &=\sum\limits_{k=1}^n B^{(1)}_{\gamma,k}\mperm_{\mm-\ee_j-\ee_{j'}-\ee_k}(U_{(\gamma-3,n)}) \\
            &=\sum\limits_{k=1}^c (1-d)\frac{(\gamma-3)!}{m_1!\cdots(m_j-1)!\cdots(m_{j'}-1)!\cdots(m_k-1)!\cdots m_n!} \\
            &\ +\sum\limits_{k=c+1}^n \frac{(\gamma-3)!}{m_1!\cdots(m_j-1)!\cdots(m_{j'}-1)!\cdots(m_k-1)!\cdots m_n!} \\
            &=k_1\left\{ \begin{array}{lcc}
             2(d-1)m_j m_{j'} &  if  & j\neq j' \ \text{and} \ j,j'\leq c \\
             2(d-1)m_j(m_j-1) &  if & j= j' \ \text{and} \ j\leq c \\
             (d-2)m_jm_{j'} &  if & j\leq c \ \text{and} \ j'> c \\
             -2m_jm_{j'} &  if & j\neq j' \ \text{and} \ j,j'> c \\
             -2m_j(m_j-1) &  if & j= j' \ \text{and} \ j> c 
             \end{array}\right. , 
\end{align*}
  which justifies the matrix $R^{(1)}$.

  Now suppose all components of $\mm$ are equal is similar. If $i=1$ or $i'=1$, then
\begin{align*}
            H_{(i,i'),(j,j')}\left(B^{(2)}\right)&=\mperm_{\mm-\ee_j-\ee_{j'}}(B^{(2)}_{ii'}) \\
            &=\mperm_{\mm-\ee_j-\ee_{j'}}(U_{(mn-2,n)}) \\
            &=k(mn-2)\left\{ \begin{array}{lcc}
             m^2 &  if  & j\neq j'  \\
             m(m-1) &  if & j= j'
            \end{array}\right. ,
\end{align*}
which justifies $Q^{(2)}$. Finally, for $R^{(2)}$ we assume $i, i' > 1$ and consider the following cases for $(j,j')$. 
    \begin{enumerate}
        \item If $j=j'=1$, then
         \begingroup\makeatletter\def\f@size{8}\check@mathfonts
\def\maketag@@@#1{\hbox{\m@th\large\normalfont#1}}%
\begin{align*}
            H_{(i,i'),(j,j')}\left(B^{(2)}\right)&=\sum\limits_{k=1}^n B^{(1)}_{1,k}mperm_{m-2e_1-e_k}(U_{(|m|-3,n)}) \\
            &= (1-n)\frac{(mn-3)!}{(m-3)!\cdots m!}  +\sum\limits_{k=2}^n \frac{(mn-3)!}{(m-2)!\cdots(m-1)!\cdots m!} \\
            &=\frac{(mn-3)!}{m!^n}m(m-1)2\left(n-1\right).
\end{align*}\endgroup
        \item If $j=1$ and $j'>1$, then $H_{(i,i'),(j,j')}\left(B^{(2)}\right)$ equals
        \begingroup\makeatletter\def\f@size{8}\check@mathfonts
\def\maketag@@@#1{\hbox{\m@th\large\normalfont#1}}%
\begin{align*}
  & (1-n)\frac{(mn-3)!}{(m-2)!\cdots(m-1)!\cdots m!}+\sum\limits_{k=2}^n \frac{(|m|-3)!}{m!\cdots (m-1)!\cdots(m-1)!\cdots(m-1)!\cdots m_n!} \\
            &=(1-n)\frac{(mn-3)!}{m!^n}m^2(m-1)+\sum\limits_{k\geq 2,k\neq j'}\frac{(mn-3)!}{m!^n}m^3 +\frac{(mn-3)!}{m!^n}m^2(m-1) \\
            &=\frac{(mn-3)!}{m!^n}m^2 \left(n-2\right).
\end{align*}\endgroup
        \item If $j\neq j'$ and $j,j'> 1$, then $H_{(i,i'),(j,j')}\left(B^{(2)}\right)$ equals
        \begingroup\makeatletter\def\f@size{8}\check@mathfonts
\def\maketag@@@#1{\hbox{\m@th\large\normalfont#1}}%
\begin{align*}
  & (1-n)\frac{(mn-3)!}{m!\cdots(m-1)!\cdots(m-1)!\cdots(m-1)!\cdots m!} \\
            &\ +\sum\limits_{k=2}^n \frac{(mn-3)!}{m!\cdots(m-1)!\cdots(m-1)!\cdots(m-1)!\cdots m!} \\
            &=(1-n)\frac{(mn-3)!}{m!^n}m^3 +2\frac{(mn-3)!}{m!^n}m^2(m-1)+\sum\limits_{k\geq 2,k\neq j,j'}\frac{(mn-3)!}{m!^n}m^3\\
            &=\frac{(mn-3)!}{m!^n}m^2 \left(-2\right).
\end{align*}\endgroup
        \item If $j=j'$ and $j> 1$, then $H_{(i,i'),(j,j')}\left(B^{(2)}\right)$ equals
        \begingroup\makeatletter\def\f@size{8}\check@mathfonts
\def\maketag@@@#1{\hbox{\m@th\large\normalfont#1}}%
\begin{align*}
        & (1-n)\frac{(mn-3)!}{m!\cdots(m-2)!\cdots(m-1)!\cdots m!}+\sum\limits_{k=2}^n \frac{(mn-3)!}{m!\cdots(m-2)!\cdots(m-1)!\cdots m!} \\
            &=(1-n)\frac{(mn-3)!}{m!^n}m^2(m-1)+\frac{(mn-3)!}{m!^n}m(m-1)(m-2)  +\sum\limits_{k\geq 2,k\neq j}\frac{(mn-3)!}{m!^n}m^2(m-1)\\
            &=\frac{(mn-3)!}{m!^n}m(m-1)(-2). \qedhere
\end{align*}\endgroup
 \end{enumerate}
\end{proof}

\begin{proof}[Proof of Theorem \ref{thm:lowerbounds}(1)]
  In light of Lemma \ref{lem:off_diagonal} and Proposition \ref{prop:hessianperm1}, it is sufficient to show that the matrices $Q^{(1)}$, $R^{(1)}$, $Q^{(2)}$, and $R^{(2)}$ are nonsingular. The idea is to express each matrix in the form $aA + bD$ where $D$ is a diagonal matrix with only one or two distinct values on the diagonal and $A$ is a block matrix of rank one or rank two. This will make it easy to verify that zero is not an eigenvalue. 

 First consider the matrices $Q^{(2)}$ and $R^{(2)}$ that arise in the case $m_1 = \dots = m_n = m$. We may assume that $m \geq 2$ because $m=1$ simply yields the permanent, and also recall that by hypothesis $n \geq 3$. Now
  \[ Q^{(2)} = k_2(mn-2)\left(m^2U_{n,n}-mI_n\right) = k_2(mn-2)m\left(mU_{n,n}-I_n\right).\]
  The eigenvalues of $U_{n,n}$ are 0 and $n$, so $mU_{n,n}-I_n$ is nonsingular for all $m > 1$. Then since $k_2 > 0$ and $mn - 2 > 0$, $Q^{(2)}$ is nonsingular.

  The matrix obtained from $R^{(2)}$ by removing its first row and column is
  \[ R_0 = k_2 m \left( -2mU_{n,n}+2I_n \right) \]
  which is again nonsingular since $m > 1$. Suppose that $\vv = (v_1, \dots, v_n)^t$ is a vector such that $R^{(2)} \vv = \mathbf{0}$. Let $\theta_1 = v_1$ and $\theta_2 = v_2 + \dots + v_n$. By taking the first row alone and summing the remaining rows, we conclude that
  \[ \left(\begin{array}{cc}
    2(m-1)(n-1) & m(n-2) \\
    (n-1)m(n-2) & -2(m-1)-2m(n-2)
\end{array}\right)\left(\begin{array}{c}
    \theta_1 \\
    \theta_2
\end{array}\right)=\left(\begin{array}{c}
    0 \\
    0
  \end{array}\right).\]
  The determinant of this system is strictly negative because the lower-right entry is negative and the other three entries are all positive (using that $n > 2$ and $m >1$.) So $\theta_1 = \theta_2 = 0$. But since $\theta_1 = v_1$, we conclude that $R_0 \vv' = 0$, where $\vv' = (v_2, \dots, v_n)^t$. Since $R_0$ is nonsingular, $\vv' = \mathbf{0}$, so $\vv = \mathbf{0}$. 

  Now consider the matrices $Q^{(1)}$ and $R^{(1)}$ that arise in the case that the components of $\mm$ are not all equal. By scaling the $i$th column of $Q^{(1)}$ by $\frac{1}{m_i}$ for each $i$, we obtain the matrix $A - I$, where $A$ is the rank-one matrix whose rows are all equal to the vector $(m_1, \dots, m_n)$. The eigenvalues of $A$ are $n$ (once) and 0 ($n-1$ times), and since $n > 1$ we conclude that $A-I$ and therefore $Q^{(1)}$ are nonsingular.

  For $R^{(1)}$, we begin by scaling the $ith$ row by $1/m_i$ for each $i$. It is sufficient to show that the resulting matrix $S$ is nonsingular, and we can express $S$ as 
  \[ \begin{pmatrix} (-2d+2) I_c & 0 \\ 0 & 2I_{n-c} \end{pmatrix} +
  \begin{pmatrix} 2d-2 & d-2 \\
    \vdots & \vdots \\
    2d-2 & d-2 \\
    d-2 & -2 \\
    \vdots & \vdots \\
    d-2 & -2
  \end{pmatrix}
  \begin{pmatrix}
    m_1 & \cdots & m_c & 0 & \cdots & 0 \\
    0 & \cdots & 0 & m_{c+1} & \cdots & m_n
  \end{pmatrix}. \] 
  If $S\xx = \zerovec$, then
  \begin{align*} \zerovec & = \begin{pmatrix} (-2d+2)x_1 \\ \vdots \\ (-2d+2)x_c \\ 2x_{c+1} \\ \vdots \\ 2x_n \end{pmatrix} +
    \begin{pmatrix} 2d-2 & d-2 \\
    \vdots & \vdots \\
    2d-2 & d-2 \\
    d-2 & -2 \\
    \vdots & \vdots \\
    d-2 & -2
  \end{pmatrix}
    \begin{pmatrix} \sum_{i=1}^c m_ix_i \\ \sum_{i=c+1}^n m_ix_i \end{pmatrix} \\
    & =
    \begin{pmatrix} (-2d+2)x_1 + (2d-2) \sum_{i=1}^c m_ix_i + (d-2) \sum_{i=c+1}^n m_ix_i \\
      \vdots \\
       (-2d+2)x_c + (2d-2) \sum_{i=1}^c m_ix_i + (d-2) \sum_{i=c+1}^n m_ix_i \\
      2x_{c+1} + (d-2) \sum_{i=1}^c m_ix_i -2 \sum_{i=c+1}^n m_ix_i \\
      \vdots \\
      2x_n + (d-2) \sum_{i=1}^c m_ix_i -2 \sum_{i=c+1}^n m_ix_i \end{pmatrix}  
  \end{align*}

  From the first $c$ coordinates all being equal to zero and hence to each other, we conclude that $x_1 = \dots = x_c =: y$, and all but one of the first $c$ equations is redundant. Similarly, from the remaining coordinates we conclude that $x_{c+1} = \dots = x_n =: z$ and all but one of the remaining equations is redundant. We thus obtain a $2 \times 2$ system
  $\left(\begin{array}{cc}
    a_{11} & a_{12} \\
    a_{12} & a_{22} 
    \end{array}\right)
  \left(\begin{array}{c}
    y \\
    z
\end{array}\right)=\left(\begin{array}{c}
    0 \\
    0
  \end{array}\right)$, where
  \[ \left(\begin{array}{cc}
    a_{11} & a_{12} \\
    a_{12} & a_{22} 
    \end{array}\right) = \left(\begin{array}{cc}
    (2d-2)\left( -1 + \sum_{i=1}^c m_i \right) & (d-2) \left( \sum_{i=c+1}^n m_i \right) \\
    (d-2) \left( \sum_{i=1}^c m_i \right) & -2 \left( -1 + \sum_{i=c+1}^n m_i \right)
    \end{array} \right). \]
    


  By definition, we have $d \geq 1$ and each $m_i > 0$, so $a_{11} > 0$. Since $n \geq c+1$ we can conlude that $\sum_{i=c+1}^n m_i \geq 1$, so $a_{22} \leq 0$. Now $a_{12}$ and $a_{21}$ have the same sign: +, -, or 0, depending on whether $d > 2$, $d < 2$, or $d=2$. So the determinant of the coefficient matrix is less than or equal to zero. For it to be zero, we would need to have $n=c+1$, $m_{c+1} = 1$ and also $d=2$. But this is not possible: we know that $m_1 > 1$, so if $n=c+1$ and $m_{c+1} = 1$, then
 \[ d=\frac{\gamma}{c m_1} = \frac{cm_1 + 1}{cm_1} < 2.\]
We conclude that $S$ and thus also $R^{(1)}$ are nonsingular.
\end{proof}
  

\bibliography{poset_polynomials}
\bibliographystyle{amsalpha}

\end{document}